\title{On Connectivity of the Facet Graphs of Simplicial Complexes}

\author{Ilan I. Newman\thanks{
Department of Computer Science, University of Haifa, Haifa, Israel. Email: {\tt ilan@cs.haifa.ac.il}. This Research was supported by The Israel Science Foundation (grant number 862/10.)}
\and Yuri Rabinovich\thanks{
Department of Computer Science, University of Haifa, Haifa,
Israel. Email: {\tt yuri@cs.haifa.ac.il}. This Research was supported
by The Israel Science Foundation (grant number 862/10.). Part of this
research was done while this author visited Mittag-Leffler Institute, Stockholm}
}

\documentclass[11pt]{article}

\usepackage{amssymb,fullpage}
\usepackage{xspace}
\usepackage{latexsym}
\usepackage{times}
\usepackage{amsfonts}
\usepackage{amsmath}
\usepackage{graphicx}

\newcommand{\ignore}[1]{}
\def \qed {\hspace*{0pt} \hfill {\quad \vrule height 1ex width 1ex depth 0pt}
 \medskip}

\newenvironment{proof}{\par\noindent{\bf Proof.}\quad}{  $\qed$ \newline}

\newcommand{\R}{\ensuremath{\mathbb R}}

\newcommand{\N}{\ensuremath{\mathbb N}}

\newcommand{\F}{\ensuremath{\mathbb F}}
\newtheorem{theorem}{Theorem}[section]

\newtheorem{definition}{Definition}[section]
\newtheorem{claim}{Claim}[section]

\newtheorem{remark}{Remark}[section]

\newtheorem{lemma}{Lemma}[section]

\newtheorem{corollary}{Corollary}[section]

\def \P {\mathcal P}
\def \C {\mathcal C}

\def \K {\mathcal S}

\def \link {{\rm link}}
\def \sign {{\rm sign}}
\def \rank {{\rm rank}}

\def \n3 {n \choose 3}
\def \without {\!\setminus\!}
\def \supp {{\rm Supp}}
\def \capp {{\rm Cap}}
\def \Im {{\rm Im}}

\begin{document}
\maketitle
\vspace{-1cm}
\begin{abstract}
The paper studies the connectivity properties of facet graphs of simplicial complexes of combinatorial 
interest. In particular, it is shown that the facet graphs of $d$-cycles, \,$d$-hypertrees and \,$d$-hypercuts are,
respectively, $(d+1)$-\,,\;$d$-\, and \,$(n-d-1)$-vertex-connected. It is also shown that the facet
graph of a $d$-cycle cannot be split into more than $s$ connected
components by removing at most $s$ vertices.  
In addition, the paper discusses various related issues, as well as an extension to cell-complexes. 
\end{abstract}

\pagenumbering{arabic}
\section{Introduction}
Graphs of convex polytopes have been studied for many decades, starting with the classical Steinitz characterization 
of the graphs of 3-polytopes~\cite{grunbaum}, experiencing a bust with the advance of the 
Simplex Method for Linear Programming, and continuing to draw a research effort in the modern era.
See, e.g., the books~\cite{grunbaum, ziegler}, and the survey~\cite{kalai} for many related results and open 
problems. One of the more famous results in the area is  Balinski Theorem~\cite{balin} from 1961, claiming that  
the graph (i.e., the $1$-skeleton) of a $d$-polytope is $(d+1)$-vertex connected. This theorem and its various 
geometrical, topological and algebraic extensions have received a considerable attention, see 
e.g.,~\cite{barn1, barn2, BjSlides, athanathos} for a very partial list of old and new related results. It has been extended 
to simple $d$-cycles in simplicial complexes in l~\cite{erannevo}, where it is shown that the 
geometric realization of such graphs in $\R^{d+1}$ is generically rigid. See also the very recent~\cite{adi}
for an algebraic treatment of graphs of simple $d$-cycles.  

In this paper we study the connectivity properties of {\em facet} graphs of simplicial complexes and, more generally, 
of cell complexes. That is, the facet graph $G_d(K)$ of a $d$-complex $K$, has a vertex for every $d$-face of $K$,
and two such vertices are connected by an edge if the corresponding faces share a $(d-1)$-face.
Since the facet graph of a convex $d$-polytope $P$ is isomorphic to the graph of its dual polytope $P^*$, the facet
graphs of convex polytopes do not require a separate study. This is not the case for simplicial complexes, where graphs
and facet graphs differ significantly. For example, it is folklore
that if a pure $d$-complex $K$ is {\em strongly connected}, i.e., its facet graph $G_d(K)$ is connected, then 
the graph of $K$ is $d$-connected. Obviously, this implication cannot be reversed, and connectivity of the graph of $K$ 
implies nothing about the connectivity of its facet graph. For another example, observe that while the graph of
a simple simplicial cycle is generically rigid by~\cite{erannevo}, its facet graph does not have to be such, 
as demonstrated by the boundary of the cross-polytope.

Motivated on one hand by the classical results about graphs of convex polytopes, and on the other hand
by the emerging combinatorial theory of simplicial complexes (see, e.g.,~\cite{lnpr} and the references therein), we study the facet graphs of the basic objects of this theory: simple $d$-cycles, $d$-hypertrees and $d$-hypercuts.
These are the higher dimensional analogues of simple cycles, spanning trees and cuts in the complete graph $K_n$.

The paper proceeds as follows. We start with building our tools, and show that a simplicial complex 
induced by at most $d$ simplices of dimension at most $d$, collapses to its $(d-2)$-skeleton. This lemma
will be generalized in the last section, and one of its variants will be shown to be equivalent
to the Homological Mixed Connectivity Theorem~\cite{Floy,BjSlides}, an elegant topological generalization of Balinski Theorem to
higher-dimensional skeletons. We shall also discuss the duality of simple cycles and hypercuts in the complete 
simplicial complex on $n$ vertices.

Next, we address the connectivity of the facet graphs of the basic combinatorial-topological 
objects. It comes, perhaps, as a little surprise that the facet graph of a simple $d$-cycle is $(d+1)$-connected.
The facet graph of a $d$-hypertree $T$ turns out to be $d$-connected, while the facet graph of a $d$-hypercut $G_d(H)$ 
is $(n-d-1)$-connected. All the results are tight.

In Section \ref{sec:5},  inspired by~\cite{klee}, we study what happens to the facet graph $G_d(Z_d)$ of a simple $d$-cycle
$Z_d$ upon removal of $s$ of its $d$-simplices. The discussion, containing a study of an extremal problem
about the Betti numbers of small $d$-complexes, leads to a somewhat unexpected conclusion 
that the remaining part of $G_d(Z_d)$ has at most $s$ components.

In the last section we study cell complexes with mild topological assumptions about the structure of the cells,
and show, among other things, that the facet graphs of a simple $d$-cycles are still $(d+1)$-connected. 

The paper employs only the very basic notions of Algebraic Topology (defined in the body of the paper), and should hopefully 
be accessible to anyone interested in Combinatorial Topology. 
\section{Preliminaries}
\label{sec:prelim}
\subsection{Basic Standard Algebraic Topology Notions}\label{prelim:basic}
We mostly use the basics of Homology Theory, beautifully presented in~\cite{Mu}.
Throughout the paper we work over a fixed finite set (universe), identified with $[n]$,
and an arbitrary fixed field~$\F$. Many of our results hold if $\F$ is replaced by any Abelian group,
however, in this paper we shall not pursue this direction.

A $d$-dimensional simplex, abbreviated as $d$-simplex, is an oriented set $\sigma \subseteq [n]$ 
with $|\sigma|=d+1$. In this paper the orientation is expressed by viewing $\sigma$ as an ordered
$(d+1)$-tuple $\sigma = (s_1,s_2,\ldots,s_{d+1})$, where $s_1 < s_2 < \ldots < s_{d+1}$.~
A {\em face} of a $\sigma$ is any (oriented) simplex supported on the subset of $V[\sigma] = \{s_1,s_2,\ldots,s_{d+1}\}$. 

A simplicial complex $K$ is a collection of simplices over $[n]$ closed under containment, 
i.e., if $\sigma \in K$, then so are all the faces of $\sigma$. As before, $\sigma \in K$ is called a face of $K$.
The dimension of $K$ is the largest dimension over all its faces. Some of the complexes discussed in this paper are {\em pure} $d$-dimensional complexes, i.e., all the maximal faces of $K$ are all of the same dimension. Such faces are called {\em facets}.  

The complete $d$-dimensional complex $K_n^d = \{\sigma \subset [n] ~|~ |\sigma| \leq d+1 \}$ contains all possible simplices
over $[n]$ of dimension at most $d$. 

 {\bf  Chains:~~} Let $K$ be a $d$-complex. We denote by  $K^{(d)}$  the set of all $d$-faces of $K$. A 
{\em $d$-chain} of $K$ is formal sum $C_d = \sum_{\sigma_i \in K^{(d)}}  c_i \sigma_i$ with $c_i \in \F$.
Alternatively, $C_d$ can viewed as a $|K^{(d)}|$-dimensional $\F$-valued vector indexed by members of
$K^{(d)}$. $d$-chains of $K$ form a linear space over $\F$.

The support $\supp(C_d)$ is the set of all $d$-simplices appearing in $C_d$ with not-zero coefficients. 
The pure simplicial simplex $K(C_d)$ associated with $C_d$ is the downwards closure of $\supp(C_d)$ with respect
to containment. A complex $K$ is said to have {\em full} $d$-skeleton
if $K^{(d)}$ contains all ${n \choose d+1}$ $d$-simplices.

{\bf The Boundary Operator:~~}
For a $d$-simplex $\sigma = (s_1,s_2,\ldots,s_{d+1})$, its  $d$-boundary is defined as a $(d-1)$-chain
$\partial_d (\sigma) = \sum_{i=1}^{d+1} (-1)^{i-1} (\sigma \without s_i)$, where $(\sigma \without s_i)$ is
an oriented facet of $\sigma$ obtained by the deletion of $s_i$. Taking a linear extension of this definition, one obtains 
a linear boundary operator $\partial_d$ from the $d$-chains over $[n]$ to the $(d-1)$-chains over $[n]$.
(Observe that for a specific complex $K$, the $d$-chains of $K$ are mapped by $\partial_d$ to the $(d-1)$-chains of $K$).
The key property of the boundary operators is that $\partial_{d-1}\partial_d = 0$.

Using the vector form of $d$-chains, $\partial_d$ is represented by a ${n \choose d} \times {n \choose d+1}$ matrix $M_d$ whose 
rows are indexed by all $(d-1)$-simplices, and columns by $d$-simplices, and $\partial C_d = M_d C_d$.  
The entries of $M_d$ are given by $M_d(\tau, \sigma) = \sign(\sigma, \tau)$, also written as $[\sigma\,:\,\tau]$,
where  $\sign(\sigma, \tau)=0$ if $\tau$ is not a facet of $\sigma$, and $\sign(\sigma, \tau)=(-1)^{i-1}$ if
$\tau$ is a facet of $\sigma$ obtained by deletion of the $i$'th element in the ordered $V[\sigma]$. 
The requirement $\partial_{d-1}\partial_d = 0$ translates to $M_{d-1} M_d = 0$ for any $d=1,2,\ldots n$.
For technical reasons, for any vertex $\sigma_i \in [n]$,  it's $(-1)$-boundary $\partial_0 (\sigma_i)$ is defined as $1$.
This extends linearly to $0$-chains. I.e., the setting is that of the reduced homology.
 
The simplicial complex $K(\partial_d C_d)$ will be often denoted by $\Delta C_d$. 

{\bf Cycles and Boundaries:~~}
A $d$-chain in $\ker (\partial_d)$ is called a {\em $d$-cycle}. The
fact $\partial_{d}\partial_{d+1} = 0$ implies that if $C_d
= \partial_{d+1}K$ then $C_d$ is a $d$-cycle. Such a cycle is called a {\em $d$-boundary} of $K$. 
The boundary of any $(d+1)$-simplex is a simple $d$-cycle of size
$d+2$, which is the smallest possible size of any $d$-cycle.  The space of $d$-cycles supported on $K$ is denoted by ${\cal Z}_d (K)$, and the space of 
$d$-boundaries supported of $K$ is denoted by ${\cal B}_d (K)$. The factor space 
${\cal Z}_d (K)~/~ {\cal B}_d (K) ~=~ \tilde H_d(K)$ is called the {\em $d$-th (reduced) homology group of $K$}. 
The dimension of $\tilde H_d(K)$ is the $d$-th Betti number of $K$, denoted by $\tilde\beta_d(K)$.

A $d$-cycle $Z$ is called {\em simple} if no other (non-zero) $d$-cycle is supported on $\supp(Z)$.
Sometimes, slightly abusing the notation, the supports of $d$-cycles will also be called $d$-cycles.

{\bf Cocycles, Coboundaries and Hypercuts:~~}
The coboundary operator $\delta^{d-1}$ is a linear operator adjoint to
$\partial_p$. It is described by the left action of $M_d$, or,
equivalently, by the right action of $M_d^T$.  For historical reasons,
both the range and the domain of $\delta^{d-1}$ are called {\em
  cochains}, and denoted $C^{d}$ and $C^{d-1}$ respectively. In this
paper, while retaining the notation, we shall not make any distinction
whatsoever between $d$-chains and $d$-cochains\footnote{While a
$d$-chain $C_d$ is regarded as a free $\F$-weighed sum of the
elements of $K^{(d)}$, the $d$-cochain $C^p$ is regarded as a
mapping from $K^{(d)}$ to $\F$.}. 

Since $M_d^T M_{d-1}^T  = (M_{d-1} M_d )^T = 0$, it holds that $\delta^{d}\delta^{d-1} = 0$.
The kernel of $\delta^{d}$,  $\ker(\delta^d)$, is the space of $d$-cocycles. A $d$-cocycle $Z^*$ is called a {\em $d$-hypercut}
if it is simple, i.e., no other non-zero cocycle is supported on $\supp(Z^*)$.

{\bf Hypertrees:~~}
A set $A$ of $d$-simplices over $[n]$ is called {\em acyclic} if there are no $d$-cycles supported on $A$. 
Equivalently, $A$ is acyclic if the columns vectors of $M_n^d$ corresponding to its elements are linearly independent over $\F$. 
Thus, it immediately follows that all maximal acyclic sets $A \subseteq K_n^d$ have the same cardinality. 
A maximal acyclic set of $d$-simplices in $K_n^d$ is called {\em
  $d$-hypertree}. Hypertrees were first introduced and studied by
Kalai~\cite{gil}. The  cardinality of every $d$-hypertree is ${n-1 \choose d}$ over any field.  

For any $d$-simplex $\sigma \in K_n^d$ and any $d$-hypertee $T$, there exists a unique $d$-cycle of the form 
$Z_d=\sigma - \capp_T(\sigma)$, where $\capp_T(\sigma) = \sum_{\zeta_i \in T} c_i \zeta_i$. Observe that 
$\partial_d \sigma = \partial_d \capp_T(\sigma)$.

A complex $K$ with full $(d-1)$-skeleton has $\tilde{H}_{d-1} (K) = 0$ if and only if $K$ contains a $d$-hypertree. 

{\bf Relevant Matroidal Notions:~~}
Given the definitions above, it is clear that $K_n^d$ defines a linear matroid ${\cal M}_d$ over $\F$, 
whose cycles correspond to supports of simple $d$-cycles as above, and whose maximal $d$-acyclic sets correspond
to $d$-hypertrees.  With slightly more effort one can show that the supports of the cycles of the {\em dual matroid} of 
${\cal M}_d$, i.e., the cocycles of ${\cal M}_d$, correspond to $d$-hypercuts. Implied by the basic matroid theory is the fact 
that every $d$-hypercut intersect every $d$-hypertree (c.f. \cite{oxley}).
\subsection{Facet Graphs}
The facet graph $G(K)=G_d(K)$ of $K$, where $K$ is a $d$-complex (or, with a slight abuse of notation, just a set of $d$-simplices, 
or even a $d$-chain), is a simple graph whose vertices correspond to the $d$-simplices in $K$, 
and two vertices form an edge if the corresponding $d$-simplices have a common $(d-1)$-dimensional face.  
Thus, each edge of $G(K)$ corresponds to a unique $(d-1)$-face of $K$. However, a $(d-1)$-face of $K$ may 
correspond to many or none of the edges of $G(K)$.

With a slight abuse of notation, we shall speak of facet graph of $d$-cycles and $d$-hypercuts, although technically they are not complexes but chains.
\section{Tools}
The following simple lemma will be at the core of many arguments to come. 
\begin{lemma}
\label{cl:d-cycle-conn}
Let $D$ be a collection of at most $d$ simplices of dimension at most $d$, and let $K(D)$ be
the corresponding simplicial complex. Then, every $(d-1)$-cycle supported 
on $K(D)$ is a $(d-1)$-boundary of $K(D)$. That is, every such cycle $Z$ is of a form 
$Z=\sum_{\sigma \in D} c_\sigma \partial\sigma$.
\end{lemma}
The combinatorial proof presented here is based on the following two claims.
In Section~\ref{sec:5} we shall establish a more general version of the lemma, 
using an algebraic-topological approach.

Call an $i$-face $\zeta$ of a simplicial complex $K$ {\em exposed} if it is contained in a unique $(i+1)$-face
$\tau$ of $K$, (in particular, such $\tau$ must be maximal).  An {\em elementary $i$-collapse} is the operation of elimination (or, alternatively, {\em collapse}) of a pair of faces 
$\zeta, \tau$ as above from $K$, resulting in a proper subcomplex of $K$. The notion of 
$i$-collapse (due to Wegner~\cite{wegner}) is frequently used in Combinatorial Topology.

\begin{claim}\label{cl:simplex}
Let $\sigma$ be a $d$-simplex, and let $T$ be a subset of faces of
$\sigma$ of dimension less than $d$, 
with $|T|\leq d-1$ . 
Let $K(T)$ be the complex defined by $T$. Call a face of $\sigma$ {\em unmarked}
if it is not in $K(T)$.
Then, there is a sequence of $(d-1)$ and $(d-2)$
elementary collapses that eliminate the $d$-face of $\sigma$, and all the unmarked  $(d-1)$-faces of $\sigma$.
\end{claim}
\begin{proof}
Since $|T|\leq d-1$, there exists an unmarked $(d-1)$-face $\tau$ of $\sigma$. Collapsing it together
with the $d$-face of $\sigma$, we arrive at $\Delta\sigma \without \{\tau\}$.

Consider the facet graph $G_{d-1}(\Delta \sigma)$. 
It is isomorphic to $K_{d+1}$, the complete graph on $d+1$ vertices, where the 
$(d-1)$-faces of $\Delta\sigma$ correspond to the vertices, and the
$(d-2)$-faces correspond (in a 1-1 manner) to the edges. 
Let $H$ be the subgraph of $G_{d-1}(\Delta \sigma)$, that is obtained
by removing all the vertices and the edges corresponding to the marked
faces. Consider all vertices of $H$ as being colored white.

We now will consider the following process of elementary
$(d-2)$-collapses that, in turn, will  color the vertices of $H$ 
blue once the corresponding $(d-1)$-faces   are  collapsed. We start with a single blue vertex
that corresponds to $\tau$. 

Observe that an edge of $H$ corresponds to a (currently) exposed $(d-2)$-face if and 
only if one of its endpoints is white, and the other is blue. Similarly, it corresponds to an (already) collapsed $(d-2)$-face if and only if both its endpoints are blue. Thus, an operation of an elementary $(d-2)$-collapse on $\Delta\sigma$ that involves only unmarked faces, can be interpreted in the terms of $H$ as follows: pick a blue vertex with a white 
neighbour, and make this neighbour blue. The goal can be equivalently restated as colouring all 
the vertices of $H$ blue.

Clearly, this is possible if and only if $H$ is connected. Indeed,  recall that $H$ is obtained from $K_{d+1}$ by removing at most $(d-1)$ vertices and edges in total (not counting the edges whose removal was caused by that of a vertex). Let $r$ be the number of removed vertices, and $q$ be the number of subsequently removed edges. Removing
$r$ vertices turns $K_{d+1}$ into $K_{d-r+1}$. The latter graph is obviously $(d-r)$-edge-connected. Therefore, removing additional $q$ edges from $K_{d-r+1}$, where $q\leq (d-1)-r$, results in a connected graph. 
\end{proof}

\begin{claim}\label{cl:collapsing}
Let $S$ be a collection  of at most $d$ simplices of dimension at most $d$, and let $K(S)$ be
the corresponding simplicial complex. Then, for $d>1$, {\em all} $d$- and 
$(d-1)$-faces of $K(S)$ can be eliminated by a series of elementary $(d-1)$- and $(d-2)$-collapses.
For $d=1$ the situation is slightly different: the unique 1-face (if any) of $K(S)$ can obviously
be eliminated by an elementary $0$-collapse, however there is no way to eliminate
the surviving 0-face(s). 
\end{claim}
\begin{proof}
The proof is by an induction on the number of $d$-simplices in $S$.

If $S$ has no $d$-simplices, then every $(d-1)$-face $\tau \in K(S)$ is a $(d-1)$-simplex in $S$.
Observe that every such $\tau$ has a (distinct) exposed $(d-2)$-facet $\zeta$ in $K(S)$. Indeed, $\tau$ has $d$ facets, while any simplex in $S \without \{\tau\}$ may un-expose at most one facet, and $|S \without \{\tau\}| < d$.
Thus, all $(d-1)$-faces of $K(S)$ can be eliminated by elementary $(d-2)$-collapses that 
eliminate pairs of faces $\zeta,\tau$ as above.

Otherwise, if $S$ contains some $d$-simplices, proceed as follows. Pick any $d$-simplex 
$\sigma \in S$, set $T  = \{\sigma \cap \xi \;|\; \xi \in S\without \{\sigma\}\}$, and mark
 the faces of $K(T)$ in $\sigma$. The assumption, $|S| \leq d$
 implies that  $|T| \leq d-1$. Hence by Claim \ref{cl:simplex}, 
the $d$-face, as well as all the unmarked $(d-1)$-faces of $\sigma$
can be collapsed by elementary $(d-1)$ and $(d-2)$ collapses. Since 
 any elementary $(d-1)$- or $(d-2)$-collapse 
in $\sigma$ that involves only the unmarked edges, can be carried out in $K(S)$ as well, 
resulting in a complex $K'(S)$ be the resulting complex.

Observe that any series of elementary $(d-1)$- and $(d-2)$-collapses 
in $K(S\without \{\sigma\})$ can also be performed in $K'(S)$. Indeed, at any stage of
collapse, an exposed $(d-1)$- or a $(d-2)$-face $\tau \in K(S\without \{\sigma\})$ is necessarily 
exposed in $K'(S)$ as well, since the faces in $K'(S) \without
K(S\without \{\sigma\})$ are of dimension $<d-1$.

Employing this observation, and applying the induction hypothesis to $K(S\without \{\sigma\})$, 
the conclusion follows.
\end{proof}
For those familiar with the properties of the collapse operation, the implication 
Claim~\ref{cl:collapsing} ~$\Longrightarrow$~ Lemma~\ref{cl:d-cycle-conn} is immediate.
For the sake of completeness, here is a simple self-contained argument:
\\
\begin{proof} {\bf (of Lemma~\ref{cl:d-cycle-conn})} 
Let $K$ be a simplicial complex, and assume that $K' = K \without \{\zeta,\tau\}$ was obtained from $K$ by an elementary $(d-2)$-collapse involving an exposed $(d-2)$-face $\zeta$, and the (unique, maximal) $(d-1)$-face $\tau$ containing it.
Then, any $(d-1)$-cycle $Z$ supported on $K$, is supported on $K'$ as well. 
In other words, the coefficient $c_\tau$ of $\tau$ in $Z$ must be $0$.
Indeed, since $\zeta$ is contained only in $\tau$, the coefficient of $\zeta$ in 
$\partial Z$ is $\sign(\tau,\zeta) \, c_\tau$, and thus $\sign(\tau,\zeta)\, c_\tau = 0$.

Next, let $K$ be a simplicial complex, and assume that $K'' = K \without \{\tau,\sigma\}$ was obtained from $K$ by an elementary $(d-1)$-collapse involving an exposed $(d-1)$-face $\tau$, and the (unique, maximal) $d$-face $\sigma$ 
containing it. Let $Z$ be a $(d-1)$ cycle supported on $K$ and let $Z' =Z - \sign(\sigma,\tau)\, c_\tau \cdot \partial \sigma$. Then,
$Z'$ is $(d-1)$-cycle supported on $K''$, and $Z$ is of a form $Z=Z' + \partial T$ for 
a $d$-chain $T= c_\tau \cdot \sigma$.

Combining the two observations, we conclude that if $R$ is obtained from $K$
by a series of elementary $(d-1)$- and $(d-2)$-collapses, then any $(d-1)$-cycle
$Z$ supported on $K$ is of the form $Z=Z'' + \partial U$, where $Z''$ is $(d-1)$-cycle supported on $R$, and $U$ is a $d$-chain on $K$. 

By Claim~\ref{cl:collapsing}, $K(D)$ collapses to a complex of dimension $<d-1$, lacking, in particular, any non-zero $(d-1)$-cycles. Thus, any $(d-1)$-cycle
$Z$ supported on $K(D)$ must be of the form $Z=\partial U$, as claimed.
\end{proof}

%
\noindent
{\bf Duality between Cycles and co-Cycles in the Complete Complex $K_n^{n-1}$ }\\
In order to discuss the structure of the facet graphs of hypercuts, it will be useful to establish a duality between  hypercuts and  simple cycles.
Such duality exists in  Matroid Theory~\cite{oxley}, and in a related, but a slightly more sophisticated form in the Algebraic Topology.
It is at the core of the important Poincar$\acute{\rm e}$ Duality and Alexander Duality.
For a relevant combinatorial exposition of the latter see~\cite{bjorner-tancer} and the references therein. 
In fact, Claim~\ref{cl:coboundary} below is an easy special case of the much more involved main result of that paper. 

Let $\Sigma$ be an $(n-1)$-simplex (seen as a complex) on the underlying space $[n]$.
I.e., $\Sigma = K_n^{n-1}$. Define a correspondence between the $(k-1)$-chains and the $(r-1)$-cochains 
of $\Sigma$, where $k+r=n$, in the following way.

For $\sigma = \langle p_1, p_2,\ldots,p_{k} \rangle$ where $1 \leq p_1 < p_2 < \ldots < p_{k} \leq n$, let $\bar{\sigma} = \langle q_1, q_2,\ldots,q_{r} \rangle$, where $1 \leq q_1 \leq q_2 \leq \cdots q_r \leq n$, and $q_j$ appears in $\bar{\sigma}$ iff it does not appear in $\sigma$. 
Set $s(\sigma)=\prod_{p_i \in \sigma} (-1)^{p_i-1}$.
The dual (signed) $(r-1)$-simplex of $\sigma$ is defined by \[
\sigma^{*} ~= ~s(\sigma) \cdot\bar{\sigma}\;.\]
Extending this definition to chains and cochains, the dual of a $(k-1)$-chain (or cochain) $C=\sum c_\sigma \sigma$ 
is defined as a $(r-1)$-cochain (respectively, chain) $C^* =  \sum c_\sigma \sigma^*$.  
The key fact about this correspondence is: 

\begin{claim}\label{cl:coboundary}
~~$(\partial_{k-1} C)^* = \delta^{r-1}\, C^*$\,.
\end{claim}
The proof appears in Appendix A.

This leads to the following lemma, to be used in the
Section~\ref{sec:cuts}, dedicated to hypercuts.  Let $k$ be a natural number in the range $[1,n]$, and let $k+r = n$.
\begin{lemma}
\label{lm:dual}
The operator $*$ defines a 1-1 correspondence between simple
$(k-1)$-cycles $Z_k$ and $(r-1)$-hypercuts $H_{r-1}$ of $K_n^{n-1}$,
given by $Z_{k-1} \mapsto Z_{k-1}^* = H_{r-1}$. Moreover, the corresponding
facet graphs $G_{k-1}(Z_{k-1})$ and $G_{r-1}(H_{r-1})$ are isomorphic.
\end{lemma}
\begin{proof}
Observe that for any chain or cochain $C$ of $K_n^{n-1}$, ~$C^{**} =
(-1)^{{{n+1}\choose 2} - n}\, C$,\; and hence  the duality map $*$ is a 1-1 correspondence between the $(k-1)$-chains and the $(r-1)$-cochains. Since by Claim~\ref{cl:coboundary}, it maps cycles to co-cycles, and co-cycles to cycles, it yields 
a 1-1 correspondence between  $(k-1)$-cycles and  $(r-1)$-co-cycles. Moreover, since it preserves containment, it yields a 1-1 correspondence between the minimal, i.e., simple, $(k-1)$-cycles and the minimal $(r-1)$-co-cycles, i.e., the $(r-1)$-hypercuts.

The isomorphism between the facet graphs of $Z_{k-1}$ and $H_{r-1}=Z_{k-1}^*$ is given by the mapping $v_\sigma \mapsto v_{\bar\sigma}$ from $V[G_{k-1}(Z_{k-1})]$ to $V[G_{r-1}(H_{r-1})]$.
Since a pair of $(k-1)$-simplices $\sigma,\zeta \in K_n^{n-1}$ share an $(k-2)$-face (i.e., are adjacent), if and only if they are both contained in a $k$-simplex $\xi\in K_n^{n-1}$, one concludes that
$\sigma,\zeta$ are adjacent iff  $\bar\sigma, \bar\zeta$ are. 
\end{proof}
%
%
\section{Basic Results}
\subsection{Connectivity of Cycles}
We are now ready to present the central results of this paper, starting 
with the $(d+1)$-connectivity of the simple $d$-cycles.
\begin{theorem}
  \label{thm:d-cycle-conn}
Let $Z$ be a simple $d$-cycle, $d\geq 1$. Then, its facet graph $G(Z)=G_d (Z)$ is $(d+1)$-connected.  
\end{theorem}
\begin{proof} Assume by contradiction that $G(Z)$ is not $(d+1)$-connected. 
Then, there exists a subset $D$ of $d$-simplexes in $\supp(Z)$,  $|D| \leq d$, 
such that the removal of the vertices corresponding to $D$ in $G(Z)$ disconnects the graph. Let
$V_1,\ldots,V_r \subset V$, $r>1$, be the vertex sets of the resulting connected components,
and let $S_1,\ldots, S_r$ be the corresponding sets of $d$-simplices in $\supp(Z)$. Finally, given that $Z=\sum c_j \sigma_j$, define $d$-chains 
$Z_i = \sum_{\sigma_j \in S_i} c_j \sigma_j$. 

By definition of $G(Z)$, different $S_i$'s have disjoint $(d-1)$-supports. 
Keeping in mind that $Z$ is a $d$-cycle, this implies that the $(d-1)$-boundaries $C_i = \partial Z_i$ are all supported on $D$. Since every $(d-1)$-boundary is a $(d-1)$-cycle, Lemma~\ref{cl:d-cycle-conn} 
applies to $C_i$'s, implying, in particular, that there exists a $d$-chain $B_1$ supported on $D$ such that 
$\partial B_1 = C_1$. Consequently, the $d$-chain $Z_1 - B_1$ is a $d$-cycle, as 
$\partial(Z_1 - B_1) = C_1 - C_1 = 0$. Since $Z_1$ and $B_1$ have disjoint supports, $Z_1 - B_1 \neq 0$. 
Also, $Z_1 - B_1$ is supported on $K_1 \cup D$, a strict
subset of $d$-faces of $Z$. This contradicts the fact that $Z$ is simple cycle, concluding
the proof.
\end{proof}
\begin{remark}
\label{rm:mixed}
The above argument yields, in fact, a slightly more robust type of connectivity than stated. 
Recall that Lemma~\ref{cl:d-cycle-conn} applies not only to $D$ as in the statement of Theorem~\ref{thm:d-cycle-conn},
but also to a union of $r$  $d$-simplices and $q$ $(d-1)$-simplices, where $r+q \leq d$. Thus, the graph $G(Z)$ remains connected after removal of any $r$ vertices and $q$ edges (or, more precisely, the edges of any $q$ cliques induced by  $(d-1)$-faces of $Z$), as long as $r+q \leq d$. 
\end{remark}
Theorem~\ref{thm:d-cycle-conn} is tight, e.g., for $d$-pseudomanifolds, i.e., simple
$d$-cycles, where every $(d-1)$-face is included in exactly two 
$d$-faces. In this case $G_d$ is $(d+1)$-regular, and thus at most $(d+1)$-connected.
For $d=1$, all simple cycles are pseudomatifolds, and thus they are exactly $2$-connected.
For $d>2$, other simple cycles exist, and it not presently clear to us whether such cycles can be more than $(d+1)$-connected, and if yes, by how much.  

Theorem~\ref{thm:d-cycle-conn} has an immediate implication on connectivity of the facet graphs of {\em $d$-biconnected} sets $S$ of $d$-complexes. This interesting notion originates in Matroid Theory, and generalizes the graph-theoretic 2-(edge)-connectivity.

Let $S$ be a set of $d$-simplices. Define the following relation on $S$:
\begin{definition}
\label{def:bi-con}
$\sigma \sim \zeta$ if there is a 
{\em simple} cycle $Z$ supported on $S$ containing both $\sigma$ and 
$\zeta$. Treating $\{\sigma, - \sigma\}$ as a simple cycle, it is also postulated 
that $\sigma \sim \sigma$.
\end{definition}

It is known from Matroid Theory~\cite{oxley} that $\sim$ is an equivalence relation.
Call $S$ {\em bi-connected} if all its $d$-simplices are $\sim$ equivalent.

\begin{corollary}
\label{cor:bi-con}
For biconnected $S$ as above, $G_d(S)$ is $(d+1)$-connected.
\end{corollary}
%
\subsection{Connectivity of Hypertrees}
Next, we establish the $(d-1)$-connectivity of  $d$-hypertrees.   
\begin{theorem}
  \label{thm:d-tree-conn}
Let $T$ be a $d$-hypertree in $K_n^d$,  $d\geq 1, ~n\geq d+2$. Then, its facet graph 
$G(T)=G_d (T)$ is $d$-connected.  
\end{theorem}
\begin{proof} 
As before, it suffices to show that for any subset $X$ of $d$-simplices of $T$,  $|X| \leq d-1$, 
the removal of the vertices corresponding to the $X$ from $G(T)$ does not disconnect the graph. 
Consider such $X$, let $V_1,\ldots,V_r \subset V$, be the vertex sets of the resulting connected 
components in $G(T)$, and let $\K_1,\ldots, \K_r$ be corresponding sets of $d$-simplices in $T$.
Let also $\K_0=X$. We shall prove that $r$ must be $1$, and thus $X$ is non-separating, as required.

For $i = 1, \ldots ,r$ let us color  all $d$-faces of $\K_i$, by  color $i$. In particular, every $d$-face of $T$  $\sigma \notin X$ has a (unique) associated color, while $X$ is colorless. 

{\it The first step} is to extend this colouring of $T$ to all $d$-simplices in  $K_n^d \without X$ in 
the following manner. 
Let $\sigma \in K_n^d \without T$~ be a $d$-simplex. As explained in Section~\ref{sec:prelim}, there
is a (unique) $d$-chain $\capp_T(\sigma)$ supported on $T$ satisfying
$\partial(\capp(\sigma)) = \partial(\sigma)$, namely $Z_\sigma = \capp_T(\sigma) - \sigma$~ is a simple $d$-cycle. 
Since any non-empty $d$-cycle is of size $\geq d+1$,  and $|X| \leq d-1$, it follows that
$\supp(Z_\sigma) \without \{ X \cup \sigma \} \subseteq T$ is not empty, and so $\capp(\sigma)$ 
must contain some coloured $d$-simplices in $T$. We claim that {\em all} such $d$-simplices must 
have the same color. This color will be assigned to $\sigma$.

Indeed, by Theorem~\ref{cl:d-cycle-conn}, the graph $G(Z_\sigma)$ is $(d+1)$-connected.
Since $|X| \leq d-1$, it remains connected after the removal of $d$ vertices corresponding
to $\{\sigma\} \cup X$. I.e., for any two coloured $d$-simplices 
$\zeta, \tau \in \capp(\sigma) \without X \subset T$, 
there exists a path $\xi_1,\xi_2,\ldots,\xi_s$ of (coloured) $d$-simplices in 
$\capp(\sigma) \without X  \subset T$
where $\xi_1=\zeta$, $\xi_s = \tau$, and every two consecutive $\xi_i, \xi_{i+1}$ share
a $(d-1)$-dimensional face. By definition of $\K_i$'s, if a pair of coloured $d$-faces of $T$
share a $(d-1)$-dimensional face, then they have the same color. Hence, all $\xi_i$'s, and in particular $\zeta$ and $\tau$, have the same color. 

Having constructed a consistent extension of the colouring of $T\without X$ to the entire 
$K_n^d \without X$, it will be convenient to extend the definition of $\K_i$'s to contain all
$d$-simplices of $K_n^d$ coloured $i$. The set $\K_0=X$ remains unaffected.

{\it The second step} is to show that any two adjacent (i.e., sharing a $(d-1)$-face) 
coloured $d$-simplices $\sigma_i, \sigma_j \in K_n^d \setminus X$ have the same
color. While in $T\setminus X$
this is immediately implied by the definition of the color classes, in
$K_n^d\setminus X$ a proof is required. 

Assume by contradiction that $\sigma_i$ and $\sigma_j$ have different colours. Let $\tau_{ij}$ be the $(d-1)$-face they share, 
and let $\zeta_i \in \capp(\sigma_i)$, $\zeta_j \in \capp(\sigma_j)$ be $d$-simplices in $T$ so that 
$\sigma_i \cap \zeta_i = \sigma_j \cap \zeta_j = \tau_{ij}$. Obviously
there are such $\zeta_i, \zeta_j$ by the definition of a cap.  Now, on one hand, $\zeta_i$ and $\zeta_j$ are adjacent,
and so, if both are colourful, it must be the same color. In addition,
this color must be the same 
as this of $\sigma_i, \sigma_j$ by consistency of the color extension. Thus, if $\sigma_i$ and $\sigma_j$ differ in
color, then at least one of $\zeta_i, \zeta_j$ must belong to $X$. In particular, $\tau_{ij}$ is a $(d-1)$-facet of
some $d$-simplex in $X$.

Let $\psi$ be the (unique) $(d+1)$-simplex containing both $\sigma_i$ and $\sigma_j$,
and let $\Delta\psi$ denote the support of its boundary. In particular, $\sigma_i, \sigma_j \in \Delta\psi$.
Consider $G_d(\Delta\psi)$, whose vertices are coloured according to the colours of the corresponding $d$-simplices, and 
the vertices and the edges corresponding respectively to $d$- and
$(d-1)$-faces of $K(X)$, are marked. As explained above, any 
two colourful vertices of $G_d(\Delta\psi)$ connected by an unmarked edge must be of the same color. Thus, showing that 
any two colourful vertices of this graph are connected by an unmarked path, will imply that there is only one color, contrary
to the assumption. 

Observe that $G_d(\Delta\psi)$ is isomorphic to $K_{d+2}$. Observe also that any $d$-simplex of $X$ may cause 
the marking of a single vertex, or, alternatively, of a single edge of $G_d(\Delta\psi)$. Since 
$|X| \leq d-1$, this amounts to at most $d-1$ vertices and edges altogether. However, by an argument 
already used in the proof of Claim~\ref{cl:d-cycle-conn}, removing all marked vertices and edges is not 
enough to disconnect $K_{d+2}$. Thus, any two coloured vertices are indeed connected by an
unmarked path, and thus the adjacent $\sigma_i$ and $\sigma_j$ must be of same color.

{\it To sum up}, we have shown so far that each 
$d$-simplex $\sigma \in K_n^d \without X$ has a well defined color,
and that every two adjacent coloured $d$-simplices have the same color.  Recall that the goal is to show that there is only one color.
Hence, to conclude the proof, it suffices to show that the facet graph $G_d(K_n^d)$ remains connected after removal of the vertices 
corresponding to the $d$-faces of $X$, i.e., that $G_d(K_n^d)$ is $(d+1)$-connected. One way of doing it is by observing that the 
$d$-skeleton of $K_n^d$ is biconnected, and then applying Corollary~\ref{cor:bi-con}.
Since $G_d(K_n^d)$ is obviously connected, by transitivity of $\sim$, it suffices
to check that any two adjacent $\sigma,\zeta$ are contained in a simple $d$-cycle. 
And indeed, they are contained in the boundary of the (unique) $(d+1)$-simplex 
containing both.

{\it In fact}, ~$G_d(K_n^d)$ is familiar in Combinatorics as
the graph of the hypersimplex polytope $\Delta_d(n)$, or as the graph
of the $(d+1)$'th slice of $n$-hypercube, where two strings are adjacent
iff they are at Hamming distance 2. See~\cite{bollo} for a relevant discussion.
The results of~\cite{athanathos} imply that this graph is $(d+1)(n-d-1)$-connected. 
The proof of~\cite{athanathos} involves an intricate geometric argument. For completeness, we attach in  
Appendix B an alternative simple combinatorial proof of this fact. 
\end{proof}

To establish the tightness of Theorem~\ref{thm:d-tree-conn},
consider first the star $T = \{\sigma \in K_n^d ~| ~ n \in
\sigma \}$. This is a $d$-hypertree:  it obviously spans
all the $d$-simplices in $K_n^d$. On the other hand, it is acyclic, as
every $\sigma \in  T$ contains an exposed face, namely $(\sigma \without n)$.  
Now, consider, e.g., the hypertree $T' = T \without \{\sigma\} \cup \{\zeta\}$ 
where $\sigma= (1, \ldots ,d,n)$, and $\zeta=(1, \ldots , d+1)$. It is easy to 
verify that $T'$ is indeed a hypertree. Observe that the $(d-1)$-face $(1,\ldots ,d)$ of $\zeta$ is exposed in $T'$, while every other $(d-1)$-face
of $\zeta$ is shared with a single $d$-simplex in $T'$. Hence, the vertex corresponding to 
$\zeta$ in $G_d(T')$ has degree $d$, implying that this graph is not $(d+1)$-connected.

Let us remark that the facet graph of a $d$-hypertree can be  more 
than $d$-connected. E.g., when $T$ is a star as above, $G_d(T)$ is obviously isomorphic to 
$G_d(K_{n-1}^{d-1})$, which is $d(n-d-2)$-connected by Theorem~\ref{thm:kn}.

Motivated by the dual definition of $r$-edge-connectivity in graphs,
namely that $G$ is $r$-edge connected if and only if  every cut (of the complete
graph) intersects $E[G]$ in at least $r$ edges, we introduce the
following definition. A $d$-complex $K$ will be called {\em $r$-connected} if
for every $d$-hypercut $H$, $|H \cap K^{(d)}| \geq r$. 

It immediately follows from Theorem~\ref{thm:d-tree-conn} that:
\begin{corollary}
\label{cor:conn}
For  $d\geq 1$, if a $d$-complex $K$ is  $r$-connected  then, $G_d(K)$ is $(d + r - 1)$-connected.
\end{corollary} 
\begin{proof}
Assume by contradiction that there is a set of $d$-simplices $D =
\{\sigma_1, \ldots , \sigma_{d+r-2} \}$ whose 
removal disconnect $G_d(K)$. Remove first $D'= \{\sigma_1, \ldots
  \sigma_{r-1} \}$ from $K$. Since by assumption, every hypercut has size
  at least $r$ in $K$, $K \setminus D'$ still contains a $d$-tree. But
  then $G_d(K \setminus D')$ is $d$-connected by Theorem
  \ref{thm:d-tree-conn}, and hence it remains connected after the removal
  of the next $d-1$ simplices in $D \setminus D'$.
\end{proof}
To demonstrate the usefulness of Corollary~\ref{cor:conn}, apply it to $K=K_n^d$. 
Since the mincut in this case is of size $n-d$ (see, e.g.,~\cite{soda}), one concludes
that $G_d(K_n^d)$ is at least $(n-1)$-connected (which is still far from being 
tight, by Theorem~\ref{thm:kn}).

Another implication of Theorem~\ref{thm:d-tree-conn} is about the connectivity of complements of 
$d$-hypercuts.
\begin{corollary}
\label{cor:comp-hyper}
Let $H \subset K_n^d$ be a $d$-hypercut, $d\geq 1$, and let $\overline{H}$ contain all $d$-simplices missed by $H$.
Then, $G_d(\overline{H})$ is $(d - 1)$-connected.
\end{corollary} 
\begin{proof} Recall that $H$, being a hypercut, is critical with respect to hitting $d$-hypertrees, i.e., it hits
every such $T$. Moreover, for any $\sigma \in H$ there exists a $d$-hypertree $T_\sigma$ such that 
$T_\sigma \cap H = \{\sigma\}$. Hence, augmenting $\overline{H}$ by any 
$\sigma \not\in \overline{H}$ makes it contain a $d$-hypertee $T_{\sigma}$. Hence
by Theorem \ref{thm:d-tree-conn}, the corresponding facet graph
$G_d(T_{\sigma})$ is
$d$-connected. Removing the extra vertex corresponding to $\sigma$ form this graph, leaves us with
$G_d(\overline{H})$, that must be $(d-1)$-connected. 
\end{proof}
While for $d=1$, Corollary~\ref{cor:comp-hyper} is trivially tight, it appears that for $d\geq 2$ it can be 
significantly strengthened. This is left as an open problem.
%
\subsection{Connectivity of Hypercuts and Cocycles}\label{sec:5}
\label{sec:cuts}
\begin{theorem}
 \label{thm:d-cocycle-conn}
Let $H$ be a $d$-hypercut, $d\geq 1$. Then, its facet graph $G_d (H)$ is $(n-d-1)$-connected.  
\end{theorem}
\begin{proof}
This is an immediate consequence of the duality result of Lemma~\ref{lm:dual}, claiming that 
$H^* \subset K_n^d$ is a simple $(n-d-2)$-cycle with $G_{d}(H) = G_{n-d-2}(H^*)$, and an application of  
Theorem~\ref{thm:d-cycle-conn} to $H^*$. 
\end{proof}
For tightness, consider the following example. Let $\tau=(1,2,\ldots,d) \in K_n^d$ be
a $(d-1)$-face. Then, the $d$-cochain 
$H_\tau = \sum_{p\not\in \tau} \sign(\tau \cup p, \tau) \cdot (\tau \cup p)$ is 
a $d$-hypercut. Its graph $G_d(H_\tau)$ is an $(n-d)$-clique, which by convention is
$(n-d-1)$-connected.

The facet graphs of cocycles that are not hypercuts, behave very differently.
For $d=1$, the cocycles are precisely the graph-theoretic cuts, and so Theorem~\ref{thm:d-cocycle-conn} applies.
For $d\geq 3$, the facet graph of cocycle can be disconnected, as exemplified
by $H_\tau + H_{\tau'}$ as above, where the Hamming distance between $\tau$ and $\tau'$ {\em as sets} is at least $3$.

For $d=2$, the answer is given by the following theorem:
\begin{theorem}
  \label{thm:2-cob}
The facet graph of a (non-empty) $2$-cocycle $Z^*$ of $K_n^2$ is $2$-connected.
\end{theorem}
\begin{proof}
It is immediate to verify the claim for $n \leq 4$, and thus we assume $n\geq 5$.
To simplify the discussion, we use the duality between cocycles and cycles, as stated in Claim~\ref{cl:coboundary} and Lemma~\ref{lm:dual}. Let $Z$ be the dual chain of $Z^*$. Then, $Z$ is a $d$-cycle, $d=n-4\geq 1$, of $K_n^{n-4}=K_{d+4}^{d}$, and $G_{2}(Z^*) = G_{d}(Z)$.

Now, $Z$, being a cycle, can be represented as $Z=\sum Z_i$, where each $Z_i$ is a simple $d$-cycle,
and $\supp(Z_i) \subseteq \supp(Z)$. The key point of the argument is
that $K_{d+4}^{d}$ is a ``narrow'' place for $d$-cycles. We claim that for
any two $Z_i, Z_j$ as above, there exists a $(d-1)$-simplex $\tau$
belonging to $K(Z_i) \cap K(Z_j)$. 

The proof is by induction on $d$. No assumption about the simplicity the $d$-cycles $Z_i,Z_j$ is made 
or required. For $d=1$, one needs to show that any two cycles in $K_5^1$ have a
common vertex. This is obvious. For general $d\geq 2$, let $Z_1,Z_2$
be two $d$-cycles. Since each of $Z_1,Z_2$ contains 
at least $(d+2)$ $d$-simplices, and  since $2(d+2) > d+4$, they share a common vertex $v$. 
Assuming $Z_1=\sum c_\ell\sigma_\ell$, let 
$
C_1 ~=~ \link_v(Z_i) ~=~ \partial \left( \sum_{\sigma_\ell \owns v} c_\ell\sigma_\ell \right).
$
$C_1$ is a nonempty $(d-1)$-boundary, and hence a nonempty $(d-1)$-cycle. Moreover, keeping in 
mind that $\partial(Z_1)=0$, we conclude that the vertex $v$ does not
appear in the vertex set $V(C_1)$. The same applies 
to the similarly defined $C_2$. Thus, $C_1,C_2$ are $(d-1)$-cycles on $(d+4)-1$ vertices. By induction 
hypothesis, they share a $(d-2)$-face $\tau'$. The desired $(d-1)$-face $\tau$ is given by $\tau=(\tau' \cup v)$.

To conclude the proof of the Theorem, consider two $d$-simplices $\sigma,\zeta \in \supp(Z)$.
If they fall in the same $Z_i$, the $(d+1)$-connectivity of $G_d(Z)$ implies that there are 
$(d+1)$ vertex-disjoint paths between the corresponding vertices $v_\sigma, v_\zeta$ in $G_d(Z)$.
Else, $\sigma\in Z_i$ and $\zeta\in Z_j$. If $Z_i$ and $Z_j$ share a common $d$-simplex $\xi$,
then, using the equivalence relation of Def.~\ref{def:bi-con}, we conclude that
$\sigma \sim \xi,~\zeta\sim \xi~\Longrightarrow \sigma \sim \zeta$, and by Cor.~\ref{cor:bi-con},
we again have at least $(d+1)$ vertex-disjoint paths. Finally, if $Z_i$ and $Z_j$ have no common $d$-simplices, by the above claim they still have a common $(d-1)$-simplex. So, there are 2 vertices in $V(G_{d}(Z_i))$, and 2 vertices in $V(G_{d}(Z_j))$, that induce a clique $K_4$ in $G_{d}(Z)$. Hence, there are at least $2$ vertex-disjoint paths between the vertices $v_\sigma, v_\zeta$ in $G_d(Z)$. 
\end{proof}
\begin{remark} We have chosen to present this proof, because it provides more information about the
structure of $G_2(Z^*)$. A simpler alternative proof would first reduce the problem to $n\leq 6$, by 
using the following argument. 

By definition of $d$-cocycles, the restriction of $Z^*$ to any subset $S\subset [n]$
is a $d$-cocycle of $K_n^d$ as well.  
Thus, for any pair of $2$-simplices  $\sigma,\zeta \in Z^*$, instead of considering the paths between the
corresponding vertices in $G_2(Z^*)$, it suffices to consider them in $G_2(Z^*|_S)$, where $S$ is the union of 
the vertex sets of $\sigma$ and $\zeta$. 
\end{remark}
%
%
\section{Extensions and Refinements}

In this section we further develop the results obtained in the previous section.
We shall make a wider use of the homology-related notions of Algebraic Topology, which are luckily 
well suited for the discussion. This will make the presentation slightly more advanced, but the benefits
will be apparent.

\subsection{Facet Graphs of Simple Cycles: Beyond Connectivity}
\label{sec:components}
What follows is a direct continuation of Theorem~\ref{thm:d-cycle-conn}.

Analogously to Klee's question about vertex graphs of convex polytopes~\cite{klee}\footnote
{
Klee studied the following question: what is the maximum possible number of the connected components in the vertex-graph of a convex $d$-polytope,
after the removal of $m$ vertices? The answer: ~it at most $1$ for $m \leq d$, ~$2$ for $m=d+1$, ~and
for a general $m$, at most the maximum possible number of facets in a convex $d$-polytope on $m$ vertices.
}, 
we ask what happens to the facet graph $G=G_d(Z)$ of a simple $d$-cycle $Z$, $d\geq 1$, after removal of a set of vertices 
$V_D \subseteq V[G]$ corresponding to a set of $d$-simplices $D \subseteq \supp(Z)$. 
The following localization theorem provides an answer to this question in terms of the topological structure of $K(D)$. 
\begin{lemma}
\label{lm:components}
Let $Z$ be a simple $d$-cycle and $D \subset \supp(Z)$. Assume that removal of $V_D$ from $G_d (Z)$ creates \;$m>1$\; 
connected components. Then, $K(D)$ contains $m$ $(d-1)$-cycles so that: {\bf (a)} any two of them have disjoint 
$(d-1)$-supports; ~{\bf (b)} any $m-1$ of them are linearly independent modulo the space of $(d-1)$-boundaries 
${\cal B}_{d-1}(K(D))$, while all $m$ of them are dependent. 
\end{lemma}

\begin{proof}
Proceeding as in the proof of Theorem~\ref{thm:d-cycle-conn}, let $V_1,\ldots,V_m \subset V$, be the vertex sets of the 
resulting connected components, and let $S_1,\ldots, S_m$ be the corresponding sets of $d$-simplices in $\supp(Z)$. Given 
that $Z=\sum c_j \sigma_j$, define the $d$-chains  $Z_i = \sum_{\sigma_j \in S_i} c_j \sigma_j$, and 
$Z_D = \sum_{\sigma_j \in D} c_j \sigma_j$. Finally, define 
the $(d-1)$-cycles $C_i = \partial Z_i$ supported on $K(D)$. This will
be the set of the desired cycles.

By definition of $G_d(Z)$, different $K(S_i)$'s have disjoint $(d-1)$-supports, and since $C_i$ is supported on $K(S_i)$,
the same applies to $C_i$'s. This establishes {\bf (a)}. 

To prove {\bf (b)}, consider, e.g., 
the first $m-1$ cycles, and assume by contradiction that for some $d$-chain $Q_d$ on $D$, and for some (not all zero)
coefficients $k_i \in \F$, it holds that ~$k_0 \cdot \partial Q_d + \sum_{i=1}^{m-1} k_i \cdot C_i = 0$.~
Let $Z' = k_0 \cdot Q_d + \sum_{i=1}^{m} k_i \cdot Z_i$ be a $d$-chain on $K(Z)$. 
Using the disjointness of $\supp(Z_i)$'s, and the disjointness of ~$\cup_{i=1}^{m-1} \supp (Z_i)$ and $D$,
it is easily verified that $Z'$ is neither $0$ nor $Z$; the latter since $\supp(Z') \cap \supp(Z_m) = \emptyset$. Moreover, by 
definition of $Z'$, it holds that $\partial  Z'=0$. This contradictions the simplicity of $Z$.

To see that $\{ C_i \}_{i=1}^m $ are linearly dependent over ${\cal B}_{d-1}(K(D))$, recall that 
$C_i = \partial Z_i$, and that $Z_D + \sum_{i=1}^m Z_i = Z$. Therefore,  
$\partial Z_D + \sum_{i=1}^m C_i = \partial Z = 0$.
\end{proof}  
As an immediate corollary to the lemma one gets:
\begin{corollary}
\label{cor:comp}
Let $Z$ be a simple $d$-cycle, and $D \subseteq \supp(Z)$.  The number
of the connected components of $G_d(Z) \setminus D$  is at most ~$1 +
\dim \tilde{H}_{d-1}(K(D))  = 1 + \tilde{\beta}_{d-1}(K(D))$. $\qed$
\end{corollary}
By Lemma~\ref{cl:d-cycle-conn}, when $|D|\leq d$, it holds that $\tilde{H}_{d-1}(K(D))=0$, and hence $\tilde{\beta}_{d-1}(K(D))=0$,
implying that there is a unique connected component.
A natural question is how large can $\tilde{\beta}_{d-1}(D)$ be as a function of $|D|$ alone, in particular when $|D|$ is large.
To prepare the necessary background for the discussion, we cite the following result.
\begin{theorem}
\label{th:LNPR}\cite{lnpr}\footnote{
Strictly speaking, most of the relevant results in that paper are formulated for \;$\rank_r(T)$,  rather than 
for \;$\dim {\cal Z}_r(K(T))$. However, the two parameters are closely related, as \;$\rank_r(T) + \dim {\cal Z}_r(K(T)) = |T^r|$.
}
Let $T$ be a set of $r$-simplices, $|T|=t$. Then, the dimension of ${\cal Z}_r(K(T))$, the space of $r$-cycles
over $K(T)$, is maximized when the set family ${\cal T} = \{\supp(\tau)\}_{\tau\in T}$\, is compressed, i.e., when 
${\cal T}$ contains the first $t$ elements in the reverse lexicographic order of the $(r+1)$-size subsets of $[n]$.
The corresponding numerical estimation is that for $t={x \choose {r+1}}$, $x\in\R$, the dimension of ${\cal Z}_r(K(T))$ 
never exceeds ${{x-1} \choose {r+1}}$. I.e.,
\[
\dim {\cal Z}_r(T) ~\leq~ t - \Omega \left(t^{{1- {1\over{r+1}}}}\right)\;.
\]
Moreover, for $t = {n\choose {r+1}}$, $n\in \N$, the optimal 
$\dim {\cal Z}_r(K(T)) = {{n-1} \choose {r+1}}$ is achieved on the $r$-skeleton of $K_n^r$. $\qed$
\end{theorem}  
This leads to the following crude estimation of $\beta_{d-1}(D)$.  Since the number of $(d-1)$-faces of $K(D)$ can be upper-bounded by $(d+1)\cdot |D|$, using Theorem~\ref{th:LNPR} one gets:
\begin{equation}\label{eq:stam}
\tilde{\beta}_{d-1}(K(D)) ~=~ \dim {\cal Z}_{d-1} (K(D)) - \dim {\cal B}_{d-1} (K(D))  ~\leq~ \dim {\cal Z}_{d-1} (K(D))
~\leq~ (d+1)\cdot |D| - \Omega(|D|^{{1- {1\over{d}}}})\;.
\end{equation}
A more accurate answer to our question was provided by Roy Meshulam:
\begin{theorem}
\label{th:roy}~\cite{roy}~
For $D$ as above, ~$\tilde{\beta}_{d-1}(K(D)) \leq d\cdot |D| -  \Omega (|D|^{{1- {1\over{d}}}})$. 
\\
On the other hand, for any sufficiently large integer $s$ of the form $s={1\over {d+1}}{n \choose {d}}$,  
there exists $D$ of size $s$ with $\tilde{\beta}_{d-1}(D) = {{n-1}\choose d} - {1\over {d+1}}{n \choose {d}}$,
which is best possible for such $s$. 
Thus, the upper bound is asymptotically tight up to the second-order terms.
\end{theorem}
\begin{proof} 
Remember that $\tilde{\beta}_{d-1}(K(D)) \;=\; \dim \tilde{H}_{d-1}(K(D)) \;=\; \dim {\cal Z}_{d-1} (K(D)) - \dim {\cal B}_{d-1} (K(D))$. 

For the upper bound, one may w.l.o.g., assume that $D$ is acyclic, i.e., $\dim {\cal B}_{d-1} (K(D)) = |D|$. 
Otherwise, if some 
$d$-simplex in $D$ is spanned by the others, removing it from $D$ effects neither the space of $(r-1)$-cycles,
nor the space of $(r-1)$-boundaries, but reduces the size of $D$. Adding some isolated $d$-simplices
to compensate the reduction in the size does not effect, again, the $(r-1)$-homology group. Thus, one gets
an acyclic set $D'$ of $d$-simplices with $|D'|=|D|$, and $\tilde{\beta}_{d-1}(K(D)) = \tilde{\beta}_{d-1}(K(D'))$. 

For an acyclic $D$, arguing as in~(\ref{eq:stam}), one gets
\[
\tilde{\beta}_{d-1}(K(D)) ~~=~~ \dim {\cal Z}_{d-1} (K(D)) - \dim {\cal B}_{d-1} (K(D)) 
~~\leq~~ \left((d+1)\cdot |D| - \Omega (|D|^{{1- {1\over{d}}}})\right) - |D| ~~=~~ \] \[
=~~ d\cdot |D| - \Omega (|D|^{{1- {1\over{d}}}})\;.
\]
For the lower bound, one may use the recent breakthrough result of Keevash~\cite{keevash}, implying, in particular,
that for any $d$, and for sufficiently large $n$ such that $(d+1)$ divides ${n \choose d}$, there exists a set $D^*$ of
$d$-simplices that covers every $(d-1)$-simplex in $K_n^{d-1}$ exactly once. Clearly, 
$|D^*| = s = {1\over {d+1}}{n \choose {d}}$. The goal is to show that for this $s$,  $\tilde{\beta}_{d-1}(K(D^*))$ is the
maximum possible.

The set $D^*$ is acyclic, and by the argument above, so is the optimal set $D^{OPT}$ of the same size. 
Thus, it suffices to argue that $\dim {\cal Z}_{d-1} (K(D^*))$ is the biggest possible.
The $(d-1)$-skeleton of $D^*$ is of size ${n \choose d} = (d+1)\cdot s$, which is the biggest possible for
any $D$ of size $s$. Finally, the $(d-1)$-skeleton of $D^*$ is $K_n^{d-1}$, which by Theorem~\ref{th:LNPR},
has the biggest possible dimension of ${\cal Z}_{d-1} (K(T))$, namely ${{n-1}\choose d}$, among 
all sets $T$ of $(d-1)$-simplices with $|T| = {n\choose d}$. Since $\,\max_{T, ~|T| = t} \,\dim {\cal Z}_{d-1} (K(T))\,$
is monotone increasing in $t$, the statement follows.
\end{proof}
Corollary~\ref{cor:comp}, i.e., part  ${\bf (b)}$ of Lemma~\ref{lm:components}, together with 
Theorem~\ref{th:roy} imply that the number of connected components obtained by removing at most $s$ 
vertices from the facet graph of a $d$-cycle is at most $ds$. Somewhat surprisingly, part ${\bf (a)}$ of 
Lemma~\ref{lm:components} yields a stronger upper bound:
\begin{theorem}
\label{th:tough}
Let $G_d(Z)$ be the facet graph of a simple $d$-cycle. Then, removing from $G_d(Z)$ any $s$ vertices, may create
at most $s$ connected components.
\end{theorem}
\begin{proof} Recall that a $(d-1)$-cycle is of size at least $d+1$. Since the $m$\, different $(d-1)$-cycles in Lemma~\ref{lm:components} are disjoint, they contain, altogether, at least $(d+1)m$\, different $(d-1)$-simplices.
On the other hand, let $D$ be the set of $d$-simplices in $\supp(Z)$ corresponding to the removed vertices.
Then, the number of $(d-1)$ faces of $K(D)$ is at most $(d+1)\cdot |D| = (d+1)s$. Thus, $(d+1)s \geq (d+1)m$,
and the conclusion follows. 
\end{proof}
The inequality $s\geq m$ is tight (for some $s$'s), as shown by the following construction achieving $s=m$.
Take a triangulation $T$ of a $d$-pseudomanifold over $\F$, (i.e., every $(d-1)$-face of $K(T)$ 
is contained in exactly two
$d$-simplices of $T$, and $T$ supports a unique nonempty $d$-cycle),
with the property that its facet graph $G_d(T)$ is bipartite. 

An example of such a triangulation of the sphere is the
$d$-cross-polytope, also known as the $d$-cocube. Its facet graph is
the graph of the cube, namely bipartite with $2^{d+1}$ vertices, and
two color classes each of size $2^d$. Another example for $d=2$ is provided by taking a torus obtained by appropriately gluing the opposite sides
of a planar $k\times k$ square, where $k\geq 4$ is even, and subdividing each $1\times 1$ square cell in it
into two triangles by drawing the North-East diagonal. 

Obviously, for such $T$, taking $D$ as all $d$-simplices in one color
class of $G_d(T)$ results is decomposing the resulting $G_d(T)
\setminus D$ into singletons.

The graph-theoretic property stated in the above theorem is called {\em toughness}. It has implications. E.g., using
Tutte's criterion for existence of a perfect matching in a graph, one
concludes via toughness that if a $d$-cycle $Z$ is of even size,
then $G(Z)$ has a perfect matching. For a survey of toughness see~\cite{tough}.    
%
%
\subsection{Cycles in Cell Complexes}
So far, we have discussed structures in simplicial complexes. In this section, we would like to discuss a class of 
axiomatically defined {\em cell complexes} that includes simplicial complexes and convex polytopes (more precisely, 
the combinatorial abstraction preserving the structure of their faces). The methods and  results
obtained  for simplicial complexes will be re-examined and
generalized. We are mostly interested in the generalizations of
Lemma~\ref{cl:d-cycle-conn} and Theorem~\ref{thm:d-cycle-conn}, in
particular we will generalize Balinski Theorem for such complexes.

Replacing simplices by {\em cells} with a specified combinatorial structure, and equipped with a boundary operator $\partial$,
gives rise to cell complexes and their homology groups. The following axioms describe the structure of the cells.
Notably, the standard assumption that the boundary of a cell is a
pseudomanifold will be replaced here by a significantly weaker assumption 
that it is a simple cycle. 
\\ \\
Formally, {\em  abstract cell complex} is   a graded poset  (partially
ordered set) $\P$ , whose elements of $\P$ will be called {\em open
  cells}. The order 
represents the cell-subcell relation.
 A (closed) cell $K(\C^o) \subseteq \P$ 
corresponding to an open cell $\C^o \in \P$, is defined as the set of
all elements that are dominated by $\C^o$ in $\P$, including $\C^o$. 

Since $\P$ is a graded poset, the {\em rank} or {\em dimension} of its elements is well defined. Define also 
$\Delta \C^o \subset \P$, the set of {\em facets} of $\C^o$ in $\P$, as the set of subcells $K(\C^o)$
of co-dimension 1.  The elements of dimension $0$ in $\P$ are
associated with the singletons in $[n]$. Moreover, $\P$ is formally extended to contain
a unique minimal element of dimension $-1$, associated with the empty
cell $\emptyset$. 

For a closed cell $C=K(C^o)$, its $0$-dim subcells are
denoted by $V(C)$, and are referred to as  its {\em vertex set}.  

A $d$-chain is a formal sum of
weighted open $d$-cells with coefficients being non-zero elements in $\F$. 

\vspace{0.1cm}
The axioms satisfied by $\P$ are as follows:
\vspace{0.1cm}
{\em 

%

{\bf A1:~} The restriction of $\cal P$ to any $K({\cal C}^o)$ is a
lattice. (I.e., every two elements in it have a unique minimal upper bound, and a
unique maximal lower bound).

{\bf A2:~} For every dimension $d \geq 0$, there is a boundary operator $\partial_d$ mapping every open $d$-cell $\C^o$ to
a $(d-1)$-chain supported on $\Delta \C^o$. In particular, $\partial_0 \{i\} = \emptyset$.
It is required that $\partial_{d+1}\partial_{d} = 0$. The operator
$\partial_d$ is linearly extended to a mapping from $d$-chains of
cells  to $(d-1)$-chains. 
A $d$-{\em cycle} is a $d$-chain $Z$ for which $\partial_d (Z)
= 0$.  $Z$ is a {\em simple} $d$-cycle if its support does not properly
contain the support of any other cycle.   

{\bf A3:~} For every open $d$-cell $\C^o$, its boundary $\partial \C^o$ is a {\em simple} cycle.
Equivalently, up to a multiplicative constant, $\partial \C^o$ is the only $(d-1)$-cycle in the closed cell $\C$.    
}

As before, $d$-chains in $\Im(\partial_{d+1})$ are called
$d$-boundaries, and by {\bf A2} they are cycles. 

\begin{definition}\label{def:199}
Call a set $T\subseteq {\cal P}$ of (open or closed) cells {\em compatible} if there exists a closed
cell in $\cal P$ containing them all. For such $T$, define the cell complex $K(T) \subseteq {\cal P}$ as the 
union of closures of cells in $T$. Observe that by axiom {\bf A1}, for any two cells in $K(T)$, there exists a unique maximal element contained in both.
\end{definition}
The whole purpose and requirement of Definition~\ref{def:199} is to
ensure the property stated in its last sentence.

Finally,  (reduced) homology groups are defined by the boundary operator $\partial$ just as 
in simplicial complexes. The facet graph $G_d(K)$ of $d$-dimensional complex $K$  
has a vertex  for each $d$-cell in $K$, and has an edge between two vertices if the corresponding 
$d$-cells have a common $(d-1)$-cell. 

\begin{claim}
\label{cl:cycle}
For $d \geq 1$, any nontrivial $d$-cycle $Z_d$ with compatible support contains at least $d+2$ $d$-cells. 
\end{claim}
\begin{proof}
The proof is by induction on $d$. For $d=0$, one needs at least two singletons for the sum of coefficients of $\emptyset$ 
to cancel out. Assume correctness for $(d-1)$. Let ${\cal C}^o$ be any
open cell in the support of $Z_d$. 
Since $\partial {\cal C}^o$ is a $(d-1)$-cycle, by inductive assumption $K({\cal C}^o)$ has $r \geq d+1$ facets 
$\Upsilon_1, \Upsilon_2, \ldots \Upsilon_r$ of dimension $d-1$. Since $\partial Z_d = 0$, for every $\Upsilon_i$ 
there exists at least one additional $d$-cell ${\cal C}^o_i \in Z$ besides ${\cal C}^o$ that contains $\Upsilon_i$. 
Observe that ${\cal C}^o_i$ may not contain any other $\Upsilon_j$, since otherwise ${\cal C}^o$ and ${\cal C}^o_i$ would have more than one common maximal subcell, contrary to {\bf A1}. 
Thus, all $\C^o_i$ are distinct, and the support $Z_d$ contains at least $d+2$ $d$-cells: ${\cal C}^o$ and $\{{\cal C}^o_i\}_{i=1}^r$.
\end{proof}
\vspace{-0.5cm}
\begin{theorem}
\label{th:general}
The facet graph $G_d(Z_d)$ of a simple $d$-cycle $Z_d$ with compatible 
support, $d \geq 1$, is $(d+1)$-connected ~(in the robust sense of
Remark~\ref{rm:mixed}).\footnote{
The assumption about a compatible support is essential. Consider,
e.g., the following set of open $2$-cells (originating from faces of convex $2$-polytopes): $C_1=(1,2,3,4,5)$ with
boundary $\partial_2 C_1= (1,2) + (2,3) + (3,4) + (4,5) - (1,5)$,  
$C_2 = (1,2,3,)$ with $\partial_2 (1,2,3) = (1,2) + (2,3) -
(1,3)$, $C_3 = (1,3,4)$ with $\partial_2 C_3 = (1,3) +(3,4) -(1,4)$, and
$C_4 = (1,4,5)$ with $\partial C_4 = (1,4) + (4,5) -(1,5)$.  Clearly, this
set of cells is not compatible. Respectively, the facet graph $G_2(Z_2)$
of the $2$-cycle $Z_2= C_1 - C_2 - C_3 - C_4$ is not $3$-connected.
} 
\end{theorem}
\begin{proof}
To facilitate the discussion, let us formulate, for every $d\geq 1$, the following two statements,
generalizing  (slightly weakened versions of)  Claim~\ref{cl:simplex} and Lemma~\ref{cl:d-cycle-conn},
respectively:
\vspace{0.1cm}\par
${\bf (I)_d}:$ ~{\em Let ${\cal C}_d$ be a closed $d$-cell, and $T$ be
  a set of at most $d-1$ open subcells of ${\cal C}_d$ of dimension
  $<d$. Let $K(T) \subset {\cal C}_d$ be the cell complex that
  corresponds to $T$, then, ~$\tilde{H}_{d-1}({\cal C}_d, K(T)) = 0$.}
\footnote{Meaning that for every relative $(d-1)$-cycle $X_{d-1}$ over ${\cal C}_d$, i.e., a $(d-1)$-chain such that
$\partial X_{d-1}=0$ outside of $K(T)$, there exists a $(d-1)$-boundary $B_{d-1}$ over ${\cal C}_d$, such 
that $B_{d-1}=X_{d-1}$ outside of $K(T)$. In fact, more can be said: due to simplicity of $\partial {\cal C}_d$, 
the boundary $B_{d-1}$ shall always be of the form $c\cdot\partial {\cal C}_d$.}

${\bf (II)_d}:$~ {\em Let $D$ be a compatible set of at most $d$ cells
  of dimension $\leq d$, and $K(D)$ the corresponding cell complex. Then, $\tilde{H}_{d-1}(K(D)) = 0$. }
\vspace{0.1cm}
\par
The argument used in deriving Theorem~\ref{thm:d-cycle-conn} from Lemma~\ref{cl:d-cycle-conn},
carries over to the present setting {\em as is}. Based on axioms {\bf A1} \& {\bf A2}, it shows that 
${\bf (II)_{d}}$ implies our Theorem.

Statement $\bf (II)_d$ will be proven by induction on $d$, with the base case $d=1$,
and two-parts induction step ~${\bf (II)_{d-1}} \Longrightarrow {\bf (I)_{d}}$, and
${\bf (I)_d} \Longrightarrow {\bf (II)_d}$.
\vspace{0.1cm}

The base case $d=1$:~ Axiom  {\bf A3} immediately implies ${\bf (II)_{1}}$.
\vspace{0.1cm}

${\bf (II)_{d-1} }\,\Longrightarrow\, {\bf (I)_{d}}$:~ 
Let ${\cal C}_d$, $T$ and $K(T)$ be as in the premise of ${\bf (I)_{d}}$.
We need to show that any relative $(d-1)$-cycle in $({\cal C}_d,K(T))$ is a relative boundary.
Consider such a relative $(d-1)$-cycle $X_{d-1}$, i.e., $\partial_{d-1}
X_{d-1} \subseteq \Delta(K(T))$. If $\partial_{d-1} X_{d-1} = 0$ there
is nothing to prove. Assume then that $X_{d-1}$ is not a cycle, but
rather just a relative cycle with respect to $K(T)$.
By induction hypothesis, ${\bf (II)_{d-1}}$ holds, implying that $\tilde{H}_{d-2}(K(T)) = 0$. Therefore,
there exists a $(d-1)$-chain $Y_{d-1}$ supported on $K(T)$, such that 
$\partial X_{d-1}  = \partial Y_{d-1}$ (here {\bf A2} is used to conclude that $\partial X_{d-1}$ is a $(d-2)$-cycle in $K(T)$). Then, $X_{d-1} - Y_{d-1}$ is a $(d-1)$-cycle in ${\cal C}_d$. 
However, by {\bf A3}, the only $(d-1)$-cycles in $\C_d$ are of the form 
$c\, \partial {\cal C}^o_d$,\, for some $c\in \F$, hence, $X_{d-1} =
Y_{d-1}  + c\,\partial {\cal C}^o_d$. 
Keeping in mind that $Y_{d-1}$ is supported on $K(T)$, one concludes that $X_{d-1}$ is a relative $d$-boundary. 
\vspace{0.1cm}

${\bf (I)_{d}} \, \Longrightarrow \,{\bf (II)_{d}}$:~ The argument is similar to the proof of 
Lemma~\ref{cl:d-cycle-conn}. Let $D$ be a set of cells as in
the premise of ${\bf (II)_{d}}$. The proof is by induction on the
number of $d$-cells in $D$.   Let $Z_{d-1}$ be
a cycle on $K(D)$. If $D$ contains no $d$-cells,
then $K(D)$ contains at most $|D|\leq d$ different $(d-1)$-cells. However, by Claim~\ref{cl:cycle},
a $(d-1)$-cycle has support of size $\geq d+1$.  Thus, in this case there are no nontrivial $(d-1)$-cycles in $K(D)$.

Hence, there exists a $d$-cell ${\cal C}_d = K({\cal C}^o_d) \in D$,
where ${\cal C}^o_d$ is an open $d$-cell.  Let $D' = D \without\{{\cal C}^o_d\}$. 
Let $T = \{\,\Psi^0 \wedge {\cal C}^o_d\,\}_{\Psi^o \in D'}$, where $\Psi^0 \wedge {\cal C}^o_d$
denotes the maximal element in $\P$ dominated by both $\Psi^0$ and
${\cal C}^o_d$, as in ${\bf A1}$. 
Observe that $T$ is a compatible set, $|T| \leq d-1$, and that the cells in $T$ are of dimension $\leq d-1$.

Applying ${\bf (I)_d}$ to ${\cal C}_d,T$,\, one concludes that
$Z_{d-1}$ is a relative boundary of ${\cal C}_d^0$ with respect to
$K(D')$. Namely, there exists a boundary $B_{d-1}$ of ${\cal C}_d$
such that $Z'_{d-1} = Z_{d-1} - B_{d-1}$ is supported on
$K(D')$. However, $D'$ has one less $d$-cell than $D$, and by
induction $Z'_{d-1} = B'_{d-1}$ for some boundary $B'_{d-1}$ in 
$D'$. Thus, $Z_{d-1} = B_{d-1} + B'_{d-1}$, which is a $(d-1)$-boundary in
$K(D)$.  
%
\end{proof}

To conclude this paper, we would like to close the circle and return to where we have started, the Balinski's Theorem.
For this, we need one more variant of ${\bf (II)_{d}}$. 
\begin{definition}
A cell complex $K$ (in particular, a closed cell) is called {\em homologically $k$-connected} if 
$\tilde{H}_i(K) = 0$ for $i=0,1,\ldots,k$.
\end{definition}
\begin{theorem}
\label{th:alexander}
Let $D$ be a compatible set of at most $d$ homologically $(d-1)$-connected cells of {\em any} dimension.\\
Then, ~$\tilde{H}_{d-1}(K(D))= 0$.
\end{theorem}
\begin{proof}
The proof is by induction on $d$, and it is almost identical to the inductive argument of the previous theorem.
Interestingly, the axiom {\bf A3} is not needed this time. We are concerned only with the modified ${\bf (I^*)_{d}}$ and ${\bf (II^*)_{d}}$, 
where the modification consists of dropping any assumptions about the dimension of the cells in $T$ and $D$, but preserving the conditions $|T|\leq d-1$, and $|D|\leq d$ respectively.  Also, in ${\bf (I^*)_{d}}$, the cell $\C$ is assumed to be of 
dimension $\geq d$.

The basis, ${\bf (II^*)_{1}}$, follows directly from the assumptions of the Theorem.

The implication ~${\bf (II^*)_{d-1} }\,\Longrightarrow\, {\bf
  (I^*)_{d}}$~ works just like in Theorem~\ref{th:general}, with the
following sole 
change:  the conclusion that the $(d-1)$-cycle $X_{d-1} - Y_{d-1}$ on ${\cal C}$ is a $(d-1)$-boundary, is now derived from the 
assumption of the Theorem that $\tilde{H}_{d-1}({\cal C})=0$.

The implication ~${\bf (I^*)_{d} }\,\Longrightarrow\, {\bf
  (II^*)_{d}}$~ still works, with induction on the number of cells in
$D$ of dimension at least $d$. 
\end{proof}

We would like to show that Theorem~\ref{th:alexander} is in fact equivalent (via Alexander duality) to the following elegant 
generalization of Balinski's Theorem due to~\cite{Floy}, further strengthened by A.~Bjorner in~\cite{BjSlides}. 
\begin{theorem}
\label{th:mixed-con} {\bf (\,Homological Mixed-Connectivity Theorem\,)} ~The boundary complex $B =\Delta P$ of a convex 
$(d+1)$-polytope remains homologically $r$-connected, with non-empty $r$-skeleton, after removal
\footnote{
Removing open faces means removing the faces themselves and their super-faces, but not their subfaces.
}
of any set $F,~|F| \leq d-r$ of its {\em open} faces, for $r=0,\ldots,d-1$. 
\end{theorem}
Clearly, $B$ is homologically $k$-connected if and only if so is its $(k+1)$-skeleton, the subcomplex of $B$ obtained by retaining 
only the faces of dimension $\leq k+1$. Thus, the case $r=0$ of the the Mixed-Connectivity Theorem is the Balinski's Theorem.
%
%
\begin{claim}
\label{cl:mixed-con} For cells corresponding to convex polytopes,  
~${\bf (II^*)}$ ~~$\,\Longleftrightarrow\,$~~ The Homological Mixed-Connectivity Theorem.
\end{claim}
\begin{proof} 
Formally, the Homological Mixed Connectivity Theorem is about vanishing of the lower homology groups of the cell 
complex $B \without U(F)$, where $U(F)$ (not a complex!) is the {\em upper} closure of $F$ in $B$ with respect to containment. 
Let $P^*$ be the dual polytope of $P$, and, respectively, let $F^*$ be the set faces dual to $F$ in $P^*$. Then, by Combinatorial 
Alexander duality for polytopes (see e.g.,~\cite{bjorner-tancer}
\footnote{
The point of~\cite{bjorner-tancer} is to provide a self-contained exposition based on the first principles.
Alternatively, the Combinatorial Alexander duality for polytopes can be derived from the topological Alexander duality,
see, e.g.,~\cite{Mu}, after providing a geometric argument that $|B \,\without\, U(F)|$, the geometric realization
of $B \,\without\, U(F)$, is homotopy equivalent to $|B^* \without K(F^*)|$.
}, in particular the discussion towards the end of the Introduction section, and the reference therein),
\[
\tilde{H}_r (B \,\without\, U(F))  ~\cong~ \tilde{H}^{d-r-1} (K(F^*))\,.
\]
Since the cohomology groups are isomorphic to the homology groups, \;$\tilde{H}^{d-r-1} (K(F^*)) \cong \tilde{H}_{d-r-1} (K(F^*))$.
Therefore, the cell complex $B \,\without\, U(F)$ has a vanishing $r$'th homology group if and only if $\tilde{H}_{d-r-1} (K(F^*))=0$. 
As the cells in $F^*$  satisfy the assumptions of Theorem~\ref{th:alexander} for any $d$, and $|F^*| = |F| \leq d-r$, this is precisely the statement of Theorem~\ref{th:alexander}. Thus, Theorem~\ref{th:alexander} implies Theorem~\ref{th:mixed-con}.

Observing that the argument is completely reversible, provided that the set $F$ is a set of faces of some polytope $P$, and that $F$
indeed satisfies this condition due to the compatibility assumption of  of Theorem~\ref{th:alexander}, the reverse implication follows as well. 

\vspace{0.1cm}
The fact that the $r$-skeleton of $B \without U(F)$ is not empty, can be shown by induction. Clearly, the most 
"destructive" set $F$ is the set of $d-r$ $0$-cells, i.e., points. Consider such $F$, and remove its points from $B$ one by one.
The link of the first point $p_1$ is a nonempty $(d-1)$-cycle $Z_{d-1}$ in the cell complex $B$, and it survives the removal of $p_1$.
Similarly, the removal of the second point in $F$ either misses $Z_d$, or reduces it to a nonempty $(d-2)$-cycle in $B$, etc. 
After the removal of the entire $F$ from $B$, a nonempty $(d-r)$-cycle $Z_{d-r}$ survives.   
\end{proof}
\\ \\
%
%
{\bf\large Acknowledgments:}~~ We are grateful to Roy Meshulam and Eran Nevo for enlightening discussions.  
\bibliographystyle{plain}

\section*{Appendix A: Proof of Claim~\ref{cl:coboundary}}
\label{ap:hyper}
The statement is that ~~$(\partial_{k-1} C)^* = \delta^{r-1}\, C^*$\,.
\\
\begin{proof}
By linearity of all the involved operators, it suffices to verify the claim for $(k-1)$-simplices.
The basic identity behind the Claim is:
\begin{equation}\label{eq:switch}
\sign(\sigma,\sigma \without p) ~~=~~ \sign(\bar{\sigma} \cup p, \bar{\sigma})\cdot (-1)^{p-1}\,,
\end{equation}
where $p\in \sigma$, and both $\sigma \without p\;$ and $\bar{\sigma} \cup p\;$ denote, with 
some abuse of notation, signed simplices ordered in the increasing order.

To verify $(\ref{eq:switch})$, assume that $p$ is the $i$'th element in $\sigma$. Then, by definition, 
$\sign(\sigma,\sigma \without p) = (-1)^{i-1}$. On the other hand, the order of 
$p$ in $\bar{\sigma} \cup p$ must be $p-i+1$, thus 
$\sign(\bar{\sigma} \cup p, \bar{\sigma}) = (-1)^{p-i}$, and (\ref{eq:switch}) follows.

The next identity is an immediate consequence of (\ref{eq:switch}):
\begin{equation}\label{eq:transform}
s(\sigma) \cdot \sign(\sigma,\sigma \without p) \cdot s(\sigma \without p) ~~=~~ \sign(\bar{\sigma} \cup p, \bar{\sigma})~.
\end{equation}
We can now establish the Claim.
\[
(\partial \sigma)^* ~=~  \left( \sum_{p\in \sigma} \sign(\sigma,\sigma\without p) \cdot (\sigma\without p) \right)^* ~=~  \sum_{p \not\in \bar{\sigma}} \sign(\sigma,\sigma\without p) \cdot [s(\sigma\without p) \cdot (\bar{\sigma} \cup p)] ~=~ 
\]
\[ =~
s(\sigma) \cdot \sum_{p \not\in \bar{\sigma}} [s(\sigma) \cdot \sign(\sigma,\sigma\without p) \cdot s(\sigma\without p)] \cdot (\bar{\sigma} \cup p) ~=~ s(\sigma) \cdot \sum_{p \not\in \bar{\sigma}} \sign((\bar{\sigma} \cup p), \bar{\sigma}) \cdot (\bar{\sigma} \cup p) ~=~ s(\sigma)\cdot \delta(\bar\sigma)
~=~ \delta(\sigma^*)~,
\]
where the fourth equality follows from (\ref{eq:transform}).
\end{proof}
\section* {Appendix B:~ Connectivity of $G_d(K_n^d)$}  
\label{sec:kn}
The facet graph of $G_d(K_n^d)$, to be denoted by $G(n,d+1)$, is the graph whose vertices correspond to
$(d+1)$-subsets of $[n]$, and a a pair of vertices forms an edge if the symmetric difference between the
corresponding sets is of size $2$.
\begin{theorem}\label{thm:kn}
For all pairs $n,d$, where $n>d+1$, and  $d>0$, the graph $G(n,d+1)$ is ~$(d+1)(n-d-1)$ connected.
\end{theorem}
\begin{proof}
The proof is by induction on the pairs $(n,d)$. Menger's Theorem will be used throughout.

For the base case $d=1$, $G(n,d+1)$ is isomorphic to the line graph of
$K_n$, which is easily verified to be $2(n-2)$ connected.
Observe also that the statement is correct for $n \leq d+3$: For
$n=d+2,$ $G_d(d+2,d+1)$ is just 
a $d+2$ clique.  
The case $n=d+3$ reduces to the case $d=1$, since two $(d+1)$-sets in $[d+3]$ are adjacent if and only if their complements,
i.e., sets of size $2$, are adjacent.  Hence we assume in what follows that $n \geq d+4$.

Separating the vertex set $V$ of $G(n,d+1)$ to $V_0$, corresponding to $(d+1)$-sets not containing $n$, and $V_1$, the rest,
we observe that $G_0 = G(n,d+1)|_{V_0}$ is isomorphic to $G(n-1,d+1)$, while $G_1 = G(n,d+1)|_{V_1}$ is isomorphic to $G(n-1,d)$.

Let $X$ be a subset of vertices of $V$, with $|X| < (d+1)(n-d-1)$.  We show that $G(n,d+1) \without X$ is connected.
\\ \\
\indent
{\bf Case 1:}~~ $|X \cap V_0| < (d+1)(n-d-2)$. \\
By induction assumption, in this case $G_0 \without X$ is connected.
Thus, either $V_1 \without X   = \emptyset$, in which case we are done, or $V_1 \without X \neq \emptyset$.
In the latter case, it suffices to show that for every $\sigma \in V_1
\without X$, there exists a path in $G\setminus X$
from $\sigma$ to a member of $V_0 \without X$.

For any subset $S \subset V_1$ let $N_i(S), ~ i=0,1$ contain all neighbours of
of $S$ in $V_i$ respectively. Note that for any $\tau \in V_1$, $|N_0(\tau)| = (n-d-1)$ and
$|N_1(\tau)| = d(n-d-1)$.

Consider $\sigma \in V_1 \setminus X$, and assume w.l.o.g., that $\sigma\without \{n\} = [d]$. 
Write $N_1(\sigma) = \cup_{j=1}^d N_{1}^j(\sigma)$, where 
$N_{1}^j (\sigma) = \{ \tau \in N_1(\sigma) | ~ j \notin \tau
\}$. Write also $A_0(\tau) = N_0(\tau) \setminus N_0(\sigma)$.
The following is a a simple observation that we will use.
\begin{claim}
 Let $\tau \in N^i_1(\sigma), \;\tau' \in N_1^j(\sigma)$. Then,  $|A_0(\tau)| = n-d-2$. \\
Furthermore, if  $i=j$, then $|A_0(\tau) \cap A_0(\tau')| = 1$, and if
$i \neq j$, then $|A_0(\tau) \cap A_0(\tau')| = 0$. $\qed$
\end{claim}
Let $r_i = |N_1^i(\sigma) \without X|$, for $i=1, \ldots d$. 

If $N_0(\sigma)$ or any of $N_0(N_1^i(\sigma) \setminus X)$ contains a member in $V_0\without X$, we are done. Otherwise,
\[
|X \cap V_1|~~\geq~~ | \cup_{i=1}^d \,N_1^i(\sigma) \cap X| ~~=~~ d(n-d-1) \,-\, \sum_1^d  r_i\,.
\]
On the other hand, in view of the Claim above,
\[
|X \cap V_0|~~\geq~~ |\,N_0(\sigma)~~\;\cup\;~~ \cup_{i=1}^d 
N_0 (N_1^i(\sigma) \setminus X) ~ \cap X\,|  =   |\,N_0(\sigma)| +
\sum_{i=1}^d ~~|\cup_{\tau \in \;N_{1}^i (\sigma) \without X} \,A_0(\tau)\,| 
~~\geq~~  \] \[(n-d-1) \,+\, \sum_1^d r_i(n-d-2) \,-\, \sum_1^d {r_i \choose 2}~.
\]
Combining the two estimations, one gets
\begin{equation}
\label{eq:last}
|X| ~~\geq~~ (d+1)(n-d-1) \;+\; \sum_i \left( r_i(n-d-3) - {r_i \choose 2} \right)\,.
\end{equation}
Since for $n \geq d+4$, it holds that~ $(n-d-3) \geq {{n-d-2} \over 2} \geq {{r_i -1}\over 2}$, the last term in~(\ref{eq:last})
is nonnegative, and thus  $|X|~\geq~ (d+1)(n-d-1)$, contradicting the assumption of Case 1. 
\\ \\
\indent
{\bf Case 2:}~~ $|X \cap V_0| ~\geq~ (d+1)(n-d-2)$, or, equivalently, ~$|X \cap V_1| ~<~ d+1$\,. \\
Since \;$d+1 < d(n-d-1)$, the graph $G_1 \without X$ is connected by the induction hypothesis.
Thus, to establish the connectivity of $G(n,d+1) \without X$ it suffices to
show that every $\sigma \in V_0 \without X$ has a
neighbour in $V_1 \without X$. Since $\sigma$ has $d+1$ neighbours in
$V_1$, and $|X \cap V_1| < d+1$ by the assumption of Case 2, the
implication follows.

This completes the proof of the statement.
\end{proof}
\end{document}